\documentclass[a4paper,leqno,12pt]{amsart}
\usepackage[leqno]{amsmath}
\usepackage{amstext,amssymb,amsthm,enumitem}
\usepackage{graphicx}
\usepackage{tikz}
\usepackage[utf8]{inputenc}
\usepackage{float}
\newtheorem*{theorem*}{Theorem}
\newtheorem{theorem}{Theorem}
\newtheorem{lemma}{Lemma}
\newtheorem{proposition}{Proposition}
\theoremstyle{definition}
\newtheorem*{prof*}{Proof}
\newtheorem*{proof*}{Proof of Theorem 1}
\newtheorem*{proofL}{Proof of Lemma 1}
\newtheorem*{proof1}{Proof of Proposition 1}
\newtheorem*{proof2}{Proof of Proposition 2}

\tikzset{dots/.append style={ultra thick, fill=none}}

\begin{document}
	
	\title[BMO spaces for nondoubling metric measure spaces]{BMO spaces for nondoubling metric measure spaces}
	
	\author{Dariusz Kosz}
	\address{ 
		\newline Faculty of Pure and Applied Mathematics
		\newline Wroc\l{}aw University of Science and Technology 
		\newline Wyb. Wyspia\'nskiego 27 
		\newline 50-370 Wroc\l{}aw, Poland
		\newline \textit{Dariusz.Kosz@pwr.edu.pl}	
	}

	\begin{abstract} In this article we study the family of $BMO^p$ spaces, $p \geq 1$, in the general context of metric measure spaces. We give a characterization theorem that allows to describe all possible relations between these spaces considered as sets of functions. Examples illustrating the obtained cases and some additional results related to the John--Nirenberg inequality are also included.
		
	
	\end{abstract}
	
	\thanks{
	%	$\\ \noindent \textit{2010 Mathematics Subject Classification.}$ Primary 42B35, 46E30.
	%	$\\ \noindent \textit{Key words:}$ $BMO$ space; metric measure space; non-doubling measure; John--Nirenberg inequality.
	The author is supported by the National Science Centre of Poland, project no. 2016/21/N/ST1/01496.	
	} 
	
	\maketitle
	\section{Introduction}
	
	$BMO$ is a function space which traditionally occurs in the literature as an object associated to the space $\mathbb{R}^d$, $d \geq 1$, equipped with the Euclidean metric and Lebesgue measure. Roughly speaking, it contains functions whose mean oscillation over a given cube $Q \subset \mathbb{R}^d$ is bounded uniformly with respect to the choice of that cube. Although $BMO$ was introduced by John and Nirenberg in \cite{JN} in the context of partial differential equations, it is also a very useful tool in harmonic analysis. One reason is that many of the operators considered there turn out to be bounded from $L^\infty$ to $BMO$ even though they are not always bounded on $L^\infty$. This, in turn, can often be used to prove the boundedness of such operators on $L^p$ for some $p \in (1, \infty)$ by using the interpolation theorem obtained by Fefferman and Stein in \cite{FS}. Another interesting thing concerns the fact that $BMO$ is dual to the Hardy space, $H^1$, which is of great use in harmonic analysis. This result was first shown by Fefferman in \cite{F}. Finally, $BMO$ functions are in close relation with other objects appearing in this field such as Carleson measures, paraproducts or commutator operators (see, e.g., \cite{C, ChS, Ch, G} for further consideration). 
	
	It is well known that most of the theory mentioned above can be developed in more general contexts that include metric measure spaces with measures which are doubling. However, the situation changes significantly if we want a measure to be completely arbitrary. Namely, many fundamental results obtained in the case of Lebesgue measure cannot be easily adapted to the non-doubling setting. In particular, there is less flexibility in using various covering lemmas in an effective way. Consequently, we have examples showing that some of the classical theorems fail to occur in certain non-doubling situations (see, e.g., \cite{A, S} for studying the weak type $(1,1)$ boundedness of the Hardy--Littlewood maximal operator), while, in contrast, some theorems can be proved for wider classes of spaces, usually requiring more complicated methods (see, e.g., \cite{NTV, T} where the boundedness of the Cauchy integral operator was studied). 
	
	Nevertheless, $BMO$ spaces for non-doubling spaces were quite successfully studied by Mateu, Mattila, Nicolau and Orobitg in \cite{MMNO}. Among other things, the authors have shown that for many Borel measures on $\mathbb{R}^d$, not necessary doubling, it is possible to define $BMO$ space in such a way as to be able to use an interpolation argument analogous to that received in \cite{FS}. On the other side, a somewhat surprising fact shown in \cite{MMNO} is that there exist measures on $\mathbb{R}^2$ for which the associated spaces $BMO$ and $BMO_b$ defined with an aid of cubes and balls, respectively, do not coincide. Another result, which will be mentioned in this paper later on, is related to some untypical behavior of the family of spaces $BMO_b^p$, $p \geq 1$, which occurs under certain conditions. In summary, there are many examples in \cite{MMNO} which illustrate that in some specific situations $BMO$ spaces may have very unusual properties. This idea also accompanies the present article. 
	
	The main motivation of this work is to study the spaces $BMO_b^p$, $p \geq 1$, considered as sets of function, in order to describe whether the natural inclusions between them are proper or not. Theorem 1 stated in Section 2 gives the characterization of all the possible cases related to this issue. Throughout the paper we deal with arbitrary metric measure spaces and hence balls determined by metrics are used to define $BMO_b^p$ spaces. From now on we omit the subscript $b$ and write $BMO^p$ instead of $BMO_b^p$.  
	
	%Recall that $\mu$ is doubling if there exists a universal constant $C > 0$ such that for every ball $B(x,r)$ we have $\mu(B(x, 2r)) \leq C \mu(B(x, r))$. 
	\section{Main result}
	
	Let $\mathbb{X} = (X, \rho, \mu)$ be a metric measure space, where $\rho$ is a metric and $\mu$ is a Borel measure such that the measure of each ball is finite and strictly positive. For a locally integrable function $f$ and an open ball $B$ we denote the average value of $f$ on $B$ by
	\begin{displaymath}
	f_B = \frac{1}{\mu(B)} \int_B f(x) d\mu(x).
	\end{displaymath}
	Then, for a parameter $p \geq 1$, we introduce the space $BMO^p(\mathbb{X})$ as the space consisting of $f$'s satisfying
	\begin{displaymath}
	\|f\|_{\ast,p} := \sup_{B \subset X} \big( \frac{1}{\mu(B)} \int_B |f(x)-f_B|^p  d\mu(x) \big)^{1/p} < \infty,
	\end{displaymath}
	where the supremum is taken over all balls contained in $X$.  
	We keep to the rule that two functions are identified if they differ by a constant. With this additional assumption $\| \cdot \|_{\ast,p}$ satisfies the norm properties and thus $BMO^p(\mathbb{X})$ can be viewed as a Banach space (it is a mathematical folklore that $BMO^p(\mathbb{X})$ is complete in any setting). If $p=1$, then we will usually write shortly $BMO(\mathbb{X})$ or $\|f\|_\ast$ instead of $BMO^1(\mathbb{X})$ or $\|f\|_{\ast, 1}$.
	
	%Obviously, if $p = 1$, then the space defined above can be seen as a natural counterpart of the standard $BMO$ space in the general situation of metric measure spaces (the only difference is that here we use balls instead of cubes). Hence we will usually write shortly $BMO(\mathbb{X})$ or $\|f\|_\ast$ instead of $BMO^1(\mathbb{X})$ or $\|f\|_{\ast, 1}$.
	%%(it is easy to show that $BMO^p(\mathbb{X})$ is complete by using the completeness of the space $L^{p}(\mathbb{X})$)
	
	Recall that by using H\"older's inequality, for $1 \leq p_1 < p_2 < \infty$, we have $\| \cdot \|_{\ast, p_1} \leq \| \cdot \|_{\ast, p_2}$ and hence $BMO^{p_2}(\mathbb{X}) \subset BMO^{p_1}(\mathbb{X})$. Consequently, if $BMO^{p_1}(\mathbb{X})$ and $BMO^{p_2}(\mathbb{X})$ coincide as sets, then the corresponding norms are equivalent. In fact, this is the case when $\mu$ is doubling, that is $\mu(B(x, 2r)) \leq C \mu(B(x, r))$ with a constant $C > 0$ independent of $x \in X$ and $r > 0$. Indeed, one can obtain that all the spaces $BMO^p(\mathbb{X})$, $p \geq 1$, coincide by using the John--Nirenberg inequality which is true for spaces with the doubling condition (see, e.g.,  \cite[Theorem A, p. 563]{MMNO}). However, the John--Nirenberg inequality fails to occur in general. Moreover, in \cite{MMNO} the authors were able to construct a (non-doubling) space $\mathbb{X}$ for which there exists $f \in BMO(\mathbb{X})$ such that $f \notin BMO^p(\mathbb{X})$, $p>1$. Here we go further and describe precisely which types of relations between the spaces $BMO^p(\mathbb{X})$, $p \geq 1$, are possible to occur. Namely, we prove the following.
	
	\begin{theorem}
		Let $\mathbb{X} = (X, \rho, \mu)$ be a metric measure space. Then we have one of the three possibilities:
	\begin{enumerate}[label=(\alph*)]
		\item all the spaces $BMO^p(\mathbb{X})$, $p \geq 1$, coincide,
		\item there exists $p_0 > 1$ such that $BMO^p(\mathbb{X})$ coincides with $BMO(\mathbb{X})$ if $p < p_0$ and $BMO^{p_1}(\mathbb{X}) \subsetneq BMO^{p_2}(\mathbb{X})$ for any $1 \leq p_1 < p_2 < \infty$ if $p_2 \geq p_0$,
		\item there exists $p_0 \geq 1$ such that $BMO^p(\mathbb{X})$ coincides with $BMO(\mathbb{X})$ if $p \leq p_0$ and $BMO^{p_1}(\mathbb{X}) \subsetneq BMO^{p_2}(\mathbb{X})$ for any $1 \leq p_1 < p_2 < \infty$ if $p_2 > p_0$.
	\end{enumerate}
	Conversely, for each of the cases described above and for any permissible choice of $p_0$ (while considering one of the last two cases) we can construct $\mathbb{X}$ for which the associated spaces $BMO^p(\mathbb{X})$, $p \geq 1$, realize the desired properties.
	\end{theorem}
	
	The proof of Theorem 1 is placed in Section 3 and it is based on certain results of a rather technical nature which are proved later on.
	
	\section{Proof of Theorem 1}
	
	In this section we prove Theorem 1. To do this we use two ingredients which we formulate here and prove in Sections 4 and 5, respectively. The first one is the following.
	
	\begin{lemma} Let $\mathbb{X} = (X, \rho, \mu)$ be a metric measure space. If $BMO^{p_1}(\mathbb{X}) \subsetneq BMO^{p_2}(\mathbb{X})$ for some $1 \leq p_1 < p_2 < \infty$, then for any $\alpha > 1$ we have $BMO^{\alpha p_1}(\mathbb{X}) \subsetneq BMO^{\alpha p_2}(\mathbb{X})$.
	\end{lemma}
	
	The second thing we need is to find a suitable family of spaces $\mathbb{X}$ for which some specific relations between the associated spaces $BMO^p(\mathbb{X})$, $p \geq 1$, occur. The process of constructing such spaces is the most technical part of this article. We obtain two complementary propositions stated below.
	
	\begin{proposition}
	Fix $p_0 > 1$. There exists a space $\mathbb{X}$ such that $BMO^p(\mathbb{X})$ coincides with $BMO(\mathbb{X})$ if and only if $p < p_0$. 
	\end{proposition}
	
	\begin{proposition}
		Fix $p_0 \geq 1$. There exists a space $\mathbb{X}$ such that $BMO^p(\mathbb{X})$ coincides with $BMO(\mathbb{X})$ if and only if $p \leq p_0$. 
	\end{proposition}
	
	Now, Theorem 1 follows easily from the results mentioned above.
	
	\begin{proof*}
		Let $\mathbb{X}$ be a metric measure space. Denote 
		\begin{displaymath}
		p_0 = \sup\{p \in [1, \infty) \colon BMO^p(\mathbb{X}) = BMO(\mathbb{X})\}.
		\end{displaymath}	
		The case $p_0 = \infty$ corresponds to $(a)$. Thus, assuming $p_0 < \infty$, we have two possibilities: $BMO^{p_0}(\mathbb{X})$ coincides with $BMO(\mathbb{X})$ or not. We analyze only the first one, which corresponds to the case $(c)$ from Theorem 1 (the second one can be considered in a similar way). Obviously, we have that $p_0 \geq 1$ and $BMO^{p}(\mathbb{X})$ coincides with $BMO(\mathbb{X})$ for each $p \leq p_0$. Now, take any $1 \leq p_1 < p_2 < \infty$ with $p_2 > p_0$. If $p_1 \leq p_0$, then $BMO^{p_1}(\mathbb{X}) \subsetneq BMO^{p_2}(\mathbb{X})$ holds by the definition of $p_0$. On the other hand, if $p_1 > p_0$, then there exists $\alpha > 1$ such that $p_1 / \alpha \leq p_0 < p_2/\alpha$. Hence, for that $\alpha$, we have $BMO^{p_1/ \alpha}(\mathbb{X}) \subsetneq BMO^{p_2/\alpha}(\mathbb{X})$ and by using Lemma 1 we conclude that $BMO^{p_1}(\mathbb{X}) \subsetneq BMO^{p_2}(\mathbb{X})$.
		
		The second part of Theorem 1 can be deduced by using the class of spaces obtained in Propositions 1 and 2 which exhausts all the possibilities associated with the cases $(b)$ and $(c)$. Since the case $(a)$ can be simply realized by any metric measure space satisfying the doubling condition, we receive the full characterization of all possible relations between the spaces $BMO^p(\mathbb{X})$, $p \geq 1$. $\raggedright \hfill \qed$
	\end{proof*}
	
	\section{Proof of Lemma 1}
	
	This section is entirely devoted to the proof of Lemma 1. It is worth mentioning here that it is possible to formulate the lemma in a more general form than the one presented in the previous section. Namely, the proof does not rely on the fact that balls were used to define the spaces $BMO^p(\mathbb{X})$, $p \geq 1$. Thus, the conclusion remains true if one considers the spaces $BMO^p(\mathbb{X})$ introduced with an aid of an arbitrary base, that is a fixed family of subsets of $X$, instead.
	
	\begin{proofL}
		Suppose that $BMO^{p_1}(\mathbb{X}) \subsetneq BMO^{p_2}(\mathbb{X})$ for some $1 \leq p_1 < p_2 < \infty$ and fix $\alpha >1$. We begin with a simple observation that it suffices to find a sequence $\{g_N\}_{N=1}^\infty$ satisfying $\|g_N\|_{\ast, \alpha p_1} \leq C$ uniformly in $N$ and $\lim_{N \rightarrow \infty} \|g_N\|_{\ast, \alpha p_2} = \infty$. 
		
		%Indeed, let us assume that $BMO^{\alpha p_1}(\mathbb{X}) = BMO^{\alpha p_2}(\mathbb{X})$ and consider the identity map $\mathbb{I} \colon BMO^{\alpha p_2}(\mathbb{X}) \rightarrow BMO^{\alpha p_1}(\mathbb{X})$. Notice that, with our assumption, $g_N \in BMO^{\alpha p_2}(\mathbb{X})$, $N \in \mathbb{N}$, and $\mathbb{I}$ is a continuous bijection between Banach spaces. We use the bounded inverse theorem to deduce that $\mathbb{I}^{-1}$ is also continuous, which in turn easily leads to a contradiction since
		%\begin{displaymath}
		%\lim_{N \rightarrow \infty} \frac{\| \mathbb{I}^{-1} (g_N) \|_{\ast, p_2}}{\|g_N \|_{\ast, p_1}} = \lim_{N \rightarrow \infty} \frac{\| g_N \|_{\ast, p_2}}{\|g_N \|_{\ast, p_1}} = \infty.
		%\end{displaymath}
		
		Take $f \in BMO^{p_1}(\mathbb{X}) \setminus BMO^{p_2}(\mathbb{X})$ and write $f = f_1 + i f_2$, where $f_1$ and $f_2$ are real-valued functions. Observe that at least one of the functions $f_i$, $i \in \{1,2\}$, also lies in $BMO^{p_1}(\mathbb{X}) \setminus BMO^{p_2}(\mathbb{X})$. Therefore, we can assume $f$ to be real-valued. 
		
		Consider an arbitrary $N \in \mathbb{N}$ and choose a ball $B_N \subset X$ such that
		\begin{equation}\label{a1}
		\frac{1}{\mu(B_N)} \int_{B_N} |f - f_{B_N}|^{p_2} \, d\mu \geq N.
		\end{equation}
		Then take $f_N = f - f_{B_N}$ and introduce $g_N$ by
		\begin{displaymath}
		g_N(x) = {\rm sgn} (f_N(x)) \cdot |f_N(x)|^{1/\alpha}.
		\end{displaymath}
		
		Our first goal is to show that $\|g_N\|_{\ast, \alpha p_1} \leq C$ uniformly in $N$. It will be convenient at this point to notice that we have
		\begin{align}\label{a2}
		\begin{split}
		\frac{1}{\mu(B)} \int_B |h - h_B|^p \, d\mu & \leq \frac{1}{\mu(B)^2} \int_B \int_B |h(x) - h(y)|^p \, d\mu(x) \, d\mu(y) \\ 
		& \leq \frac{2^p}{\mu(B)} \int_B |h - h_B|^p \, d\mu,
		\end{split}
		\end{align}
		for any $p \geq 1$, $B \subset X$ and $h$ which is locally integrable. Take an arbitrary ball $B$ and note that (\ref{a2}) implies
		\begin{equation}\label{a3}
		\frac{1}{\mu(B)^2} \int_B \int_B |f_N(x) - f_N(y)|^{p_1} \, d\mu(x) \, d\mu(y) \leq 2^{p_1} \|f_N\|_{\ast, p_1}^{p_1} = 2^{p_1} \|f\|_{\ast, p_1}^{p_1}.
		\end{equation} 
		We would like to receive a similar estimate for $g_N$ and $\alpha p_1$ instead of $f_N$ and $p_1$, respectively. Take any two points $x$ and $y$ contained in $B$. If $g_N(x)$ and $g_N(y)$ are of the same sign, then
		\begin{equation*}
		|g_N(x) - g_N(y)|^{\alpha p_1} = \big| \, |f_N(x)|^{1/ \alpha} - |f_N(y)|^{1/ \alpha} \, \big|^{\alpha p_1} \leq    |f_N(x) - f_N(y)|^{p_1}.
		\end{equation*} 
		On the other hand, if, for instance, $g_N(x) > 0$ and $g_N(y) \leq 0$, then we obtain
		\begin{align*}
		\begin{split}
		|g_N(x) - g_N(y)|^{\alpha p_1} \leq 2^{\alpha p_1} (g_N(x)^{\alpha p_1} + (-g_N(y))^{\alpha p_1}) & = 2^{\alpha p_1} (f_N(x)^{p_1} + (-f_N(y))^{p_1}) 
		\\ & \leq 2^{\alpha p_1} |f_N(x) - f_N(y)|^{p_1}.
		\end{split}
		\end{align*} 
		Combining (\ref{a3}) with the last two estimates gives
		\begin{displaymath}
		\frac{1}{\mu(B)^2} \int_B \int_B |g_N(x) - g_N(y)|^{\alpha p_1} \, d\mu(x) \, d\mu(y) \leq 2^{(1 + \alpha)p_1} \|f\|_{\ast, p_1}^{p_1},
		\end{displaymath} 
		which, by using (\ref{a2}) one more time, results in the desired inequality $\|g_N\|_{\ast, \alpha p_1} \leq 2^{1+ \alpha} \|f\|_{\ast, p_1}$.
		
		Now, the only thing left to do is to estimate $\|g_N\|_{\ast, \alpha p_2}$ from below. Namely, for a fixed $M > 0$ we take $N$ satisfying
		\begin{equation}\label{a4}
		2^{-\alpha p_2} N - 2^{\alpha p_2} (M+1)^{\alpha p_2} \geq M,
		\end{equation}  
		and show that
		\begin{equation}\label{a5}
		\frac{1}{\mu(B_N)} \int_{B_N} |g_N - (g_N)_{B_N}|^{\alpha p_2} \, d\mu \geq M.
		\end{equation}
		We consider two cases, $|(g_N)_{B_N}| \leq M+1$ and $(g_N)_{B_N} < -M-1$ (in the case $(g_N)_{B_N} > M+1$ one can replace $f_N$ and $g_N$ by $-f_N$ and $-g_N$, respectively). If $|(g_N)_{B_N}| \leq M+1$, then we use the following estimates: for $x \in B_N$ such that $|g_N(x)| > 2(M+1)$,
		\begin{equation}\label{a6}
		|g_N(x) - (g_N)_{B_N}|^{\alpha p_2} \geq 2^{-\alpha p_2} |g_N(x)|^{\alpha p_2} = 2^{-\alpha p_2} |f_N(x)|^{p_2},
		\end{equation}
		and for $x \in B_N$ such that $|g_N(x)| \leq 2(M+1)$,
		\begin{equation}\label{a7}
		|g_N(x) - (g_N)_{B_N}|^{\alpha p_2} \geq 0 \geq |g_N(x)|^{\alpha p_2} - 2^{\alpha p_2} (M+1)^{\alpha p_2} = |f_N(x)|^{p_2} - 2^{\alpha p_2} (M+1)^{\alpha p_2}.
		\end{equation}
		Applying (\ref{a1}), (\ref{a4}), (\ref{a6}) and (\ref{a7}) we obtain
		\begin{align}\label{a8}
		\begin{split}
		\int_{B_N} |g_N - (g_N)_{B_N}|^{\alpha p_2} \, d\mu & \geq
		2^{- \alpha p_2} \int_{B_N} |f_N|^{p_2} \, d\mu - 2^{\alpha p_2} (M+1)^{\alpha p_2} \mu(B_N) \\
		&  \geq \Big( 2^{- \alpha p_2} N - 2^{\alpha p_2} (M+1)^{\alpha p_2} \Big) \mu(B_N) \\ & \geq M \mu(B_N).
		\end{split}
		\end{align} 
		
		In turn, if $(g_N)_{B_N} < -M-1$, then, equivalently,
		\begin{displaymath}
		\int_{B_N} (f_N - g_N) \, d\mu > (M+1) \mu(B_N).
		\end{displaymath}
		Let $U_N = \{x \in B_N \colon g_N(x) \geq 1\}$. Observe that for any $y \in B_N \setminus U_N$ we have $f_N(y) - g_N(y) \leq 1$ and hence
		\begin{equation}\label{a9}
		\int_{U_N} (f_N - g_N) \, d\mu > M \mu(B_N).
		\end{equation}
		Therefore, by using the definition of $U_N$, the fact that $(g_N)_{B_N} < 0$ and (\ref{a9}) we receive
		\begin{align}\label{a10}
		\begin{split}
		\int_{B_N} |g_N - (g_N)_{B_N}|^{\alpha p_2} \, d\mu \geq \int_{U_N} g_N^{\alpha p_2} \, d\mu & = \int_{U_N} f_N^{p_2} \, d\mu \\ & \geq \int_{U_N} (f_N - g_N) \, d\mu \\ & > M \mu(B_N).
		\end{split}
		\end{align} 
		Finally, (\ref{a5}) is a consequence of (\ref{a8}) and (\ref{a10}).
	\end{proofL}
	
	\section{Test spaces}
	
	In this section we present a simple method of constructing metric measure spaces $\mathbb{X} = (X, \rho, | \, \cdot \, |)$ with specific properties of the associated spaces $BMO^p(\mathbb{X})$, $p \geq 1$. Here $| \, \cdot \, |$ refers to the counting measure which is the only measure that will be considered in Sections 5 and 6. Before reading the exact description of the constructed spaces, it may be helpful to take a look at Figure 1 presented later on in this section.
	
	We use the term \textit{test space} for each $\mathbb{X}$ built in the following way. Let $M = \{m_{n,i}  \colon i=1, \dots, n ,\ n \in \mathbb{N}\}$ be a fixed triangular matrix of positive integers with $m_{1,1}=1$. Define 
	\begin{displaymath}
	X = X_M = \{x_{n,i,j} \colon j=0, \dots, m_{n,i}, \ i=1, \dots, n, \ n \in \mathbb{N}\} ,
	\end{displaymath}
	where all elements $x_{n,i,j}$ are pairwise different. By $S_{n,i}$ we denote the branch $S_{n,i} = \{ x_{n,i,0}, x_{n,i,1}, \dots , x_{n,i,m_{n,i}} \}$. Later on we use also auxilliary symbols $S_n = \cup_{i=1}^{n} S_{n,i}$, $T_n = \cup_{k=1}^{n} S_k$ and the function $\vee \colon X \times X \rightarrow \mathbb{N}$ defined by $\vee(x,y) = \min \{n \in \mathbb{N} \colon \{x,y\} \subset T_n \}$. We introduce the metric $\rho$ on $X$ determining the distance between two different elements $x$ and $y$ by the formula
	\begin{displaymath}
	\rho(x,y) = \left\{ \begin{array}{rl}
	n+\frac{1}{2} & \textrm{if } \{x,y\} = \{x_{n,n,0}, x_{n+1,1,0}\} \textrm{ for some } n \in \mathbb{N},  \\
	
	n-\frac{1}{2i+1} & \textrm{if } x_{n,i,0} \in \{x,y\} \subset S_{n,i} \textrm{ for some } 1 \leq i \leq n, n \in \mathbb{N},  \\
	
	n-\frac{1}{2i+2} & \textrm{if } \{x,y\} = \{x_{n,i,0}, x_{n,i+1,0}\} \textrm{ for some } 1 \leq i \leq n-1, n \in \mathbb{N},  \\
	
	\vee(x,y) & \textrm{otherwise. }  \end{array} \right. 
	\end{displaymath}
	At first glance, such a metric may look a little strange. However, its main advantage lies in the arrangement of balls containing exactly two points which we call \textit{pair of neighbors} later on. Moreover, any ball that cannot be covered by at least one of the sets $\mathcal{N}_x := \{x\} \cup \{y \colon y \text{ is a neighbor of } x\}$, $x \in X$, must be of the form $T_n$ or $T_n \cup \{x_{n+1, 1, 0}\}$ for some $n \geq 2$. These two properties make the associated $BMO^p(\mathbb{X})$ spaces easier to deal with. Figure 1 shows a model of the space $(X, \rho)$ with particular emphasis on the fact that each two neighboring points are connected by a solid line.
	
	\begin{figure}[H]
		\begin{tikzpicture}
		 [scale=.9,auto=left]
		% [scale=.8,auto=left,every node/.style={circle,fill,inner sep=2pt}]
		\node[style={circle,fill,inner sep=2pt}, label={[yshift=-1cm]$x_{1,1,0}$}] (k0) at (1.5,1) {};
		\node[style={circle,fill,inner sep=2pt}, label={$x_{1,1,1}$}] (k1) at (0.75,2.25)  {};
		\node[style={circle,fill,inner sep=2pt}, label={[yshift=-0.05cm]$x_{1,1,m_{1,1}}$}] (k2) at (2.25,2.25)  {};
		\node[dots] (k4) at (1.5,2.25)  {...};
		
		\node[style={circle,fill,inner sep=2pt}, label={[yshift=-1cm]$x_{2,1,0}$}] (l0) at (4,1) {};
		\node[style={circle,fill,inner sep=2pt}, label=$x_{2,1,1}$] (l1) at (3,3.25)  {};
		\node[style={circle,fill,inner sep=2pt}, label={[yshift=-0.05cm]$x_{2,1,m_{2,1}}$}] (l2) at (5,3.25)  {};
		\node[dots] (l4) at (4,3.25)  {...};
		
		\node[style={circle,fill,inner sep=2pt}, label={[yshift=-1cm]$x_{2,2,0}$}] (m0) at (8,1) {};
		\node[style={circle,fill,inner sep=2pt}, label=$x_{2,2,1}$] (m1) at (6.5,4.5)  {};
		\node[style={circle,fill,inner sep=2pt}, label={[yshift=-0.05cm]$x_{2,2,m_{2,2}}$}] (m2) at (9.5,4.5)  {};
		\node[dots] (m4) at (8,4.5)  {...};
		
		\node[style={circle,fill,inner sep=2pt}, label={[yshift=-1cm]$x_{3,1,0}$}] (n0) at (13,1) {};
		\node[style={circle,fill,inner sep=2pt}, label=$x_{3,1,1}$] (n1) at (11,6)  {};
		\node[style={circle,fill,inner sep=2pt}, label={[yshift=-0.05cm]$x_{3,1,m_{3,1}}$}] (n2) at (15,6)  {};
		\node[dots] (n4) at (13,6)  {...};
		
		\node[dots] (o) at (17,1)  {...};
		
		%\node[draw,text width=4cm] at (2,-2) {some text spanning three lines with automatic line breaks};
		
		\draw (k0) -- (l0) node [midway, fill=white, above=-22.5pt] {$\frac{3}{2}$};
		\draw (l0) -- (m0) node [midway, fill=white, above=-22.5pt] {$\frac{7}{4}$};
		\draw (m0) -- (n0) node [midway, fill=white, above=-22.5pt] {$\frac{5}{2}$};
		
		\draw (k0) -- (k2) node [midway, fill=white, left=10pt, above=-10pt] {$\frac{2}{3}$};
		\draw (l0) -- (l2) node [midway, fill=white, left=12.5pt, above=-10pt] {$\frac{5}{3}$};
		\draw (m0) -- (m1) node [midway, fill=white, left=10pt, above=-10pt] {$\frac{9}{5}$};
		\draw (m0) -- (m2) node [midway, fill=white, left=-10pt, above=-10pt] {$\frac{9}{5}$};
		\draw (n0) -- (n1) node [midway, fill=white, left=12.5pt, above=-10pt] {$\frac{8}{3}$};
		\draw (n0) -- (n2) node [midway, fill=white, left=-12.5pt, above=-10pt] {$\frac{8}{3}$};
		
		%\draw (k0) -- (k4) node [midway, fill=white, above=-9pt] {$\frac{2}{3}$};
		
		\foreach \from/\to in {k0/k1, k0/k2, l0/l1, l0/l2, n0/n1, n0/n2, m0/m1, m0/m2, k0/l0, l0/m0, m0/n0, n0/o}
		\draw (\from) -- (\to);
		
		%\draw (k0) -- (l0) node [midway, fill=white, above=-10pt] {$3/2$};
		\end{tikzpicture}
		\caption{The model of the space $(X, \rho)$.}
	\end{figure}
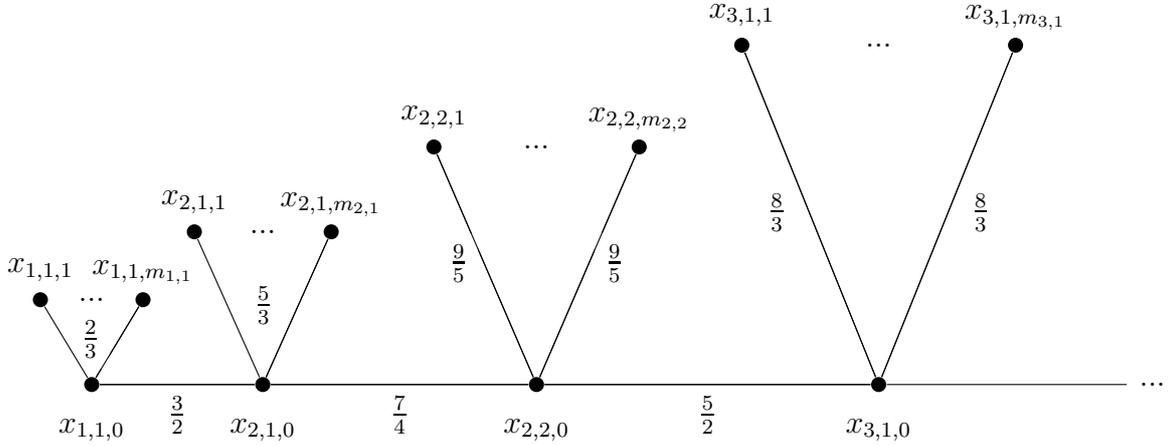
	
	Let us fix $p_0 > 1$. Our intention is to choose the matrix $M$ in such a way as to obtain that $BMO^p(\mathbb{X}) = BMO(\mathbb{X})$ if and only if $p < p_0$. We construct $M$ inductively. Namely, for each $n \geq 2$, supposing that the values $m_{k,i}$, $i = 1, \dots, k$, $k < n$, have already been chosen, we take
	\begin{equation} \tag{C1} \label{b1}
	m_{n, i} = \Big\lfloor \frac{b_{n} }{(n - i + 1)^{p_0}} - \frac{b_{n} }{(n - i + 2)^{p_0}} \Big\rfloor, \qquad i=1, \dots, n, 
	\end{equation} 
	where $\lfloor \, \cdot \, \rfloor$ is the floor function and $b_n$ is an even positive integer so large that
	\begin{equation} \tag{C2} \label{b2}
	|T_{n-1}| \leq \min \Big\{ \Big\lfloor \frac{b_{n} }{(n + 1)^{p_0}} - \frac{b_{n} }{(n + 2)^{p_0}} \Big\rfloor, \frac{b_{n}}{n^{2p_0}}\Big\}. 
	\end{equation}
	We need some auxilliary estimates. First, observe that from (\ref{b1}), (\ref{b2}) and the fact that $b_{n}$ is even it follows that $b_{n} / 2 \leq |T_{n}| \leq 2  b_{n}$. Moreover, for each $i=1, \dots, n$, we have
	\begin{equation}\label{b3}
	\frac{|S_{n, i}|}{|T_{n}|} \leq \frac{4m_{n,i}}{b_{n}} \leq 4 \, \Big( \frac{1}{(n - i + 1)^{p_0}} - \frac{1}{(n - i + 2)^{p_0}} \Big), 
	\end{equation}
	and
	\begin{equation}\label{b4}
	\frac{|S_{n, i}|}{|T_{n}|} \geq \frac{m_{n,i}}{2b_{n}} \geq \frac{1}{4} \, \Big( \frac{1}{(n - i + 1)^{p_0}} - \frac{1}{(n - i + 2)^{p_0}} \Big). 
	\end{equation}
	We are ready to prove Proposition 1.
	
	%\begin{proposition}
	%Fix $p_0 > 1$ and let $\mathbb{X} = (X, \rho, | \, \cdot \, |)$ be the test space with $M$ defined by using (\ref{b1}) and (\ref{b2}). Then for each $1 < p < p_0$ there exists $C_p > 0$ such that $\|f\|_{\ast, p} \leq C_p \|f\|_\ast$ for every $f \in BMO(\mathbb{X})$. On the other hand, there exists $g \in BMO(\mathbb{X})$ such that $g \notin BMO^{p_0}(\mathbb{X})$. 	
	%\end{proposition}
	
	\begin{proof1}
	For a fixed $p_0 > 1$ we let $\mathbb{X} = (X, \rho, | \, \cdot \, |)$ be the test space with $M$ defined by using (\ref{b1}) and (\ref{b2}). 
	
	First we show that for each $1 < p < p_0$ there exists $C_p > 0$ such that $\|f\|_{\ast, p} \leq C_p \|f\|_\ast$ for every $f \in BMO(\mathbb{X})$. Take $f \in BMO(\mathbb{X})$ and $1 < p < p_0$. Without any loss of generality we can assume that $\|f\|_\ast = 1$. Observe that then we have $|f(x) - f(y)| \leq 2$ whenever $x$ and $y$ are neighbors. Hence for each $B \subset X$ we have at least one of the two possibilities:
		\begin{enumerate}[label=(\alph*)]
			\item $B \subset \mathcal{N}_x$ for some $x \in X$ and then, by the triangle inequality, 
			\begin{displaymath}
			\max \{ |f(y) - f(z)| \colon y, z \in B\} \leq 4,
			\end{displaymath}
			\item $B$ is of the form $T_n$ or $T_n \cup \{x_{n+1,1,0}\}$ for some $n \geq 2$.	
		\end{enumerate}
		If (a) holds, then we obtain the trivial bound
		\begin{equation}\label{b5}
		\frac{1}{|B|} \sum_{x \in B} |f(x) - f_B|^p \leq 4^p.
		\end{equation}
		In turn, if (b) holds, then we fix $n \geq 2$ and assume that $B = T_{n}$ or $B = T_{n}\cup \{x_{n+1,1,0}\}$. Let $E'_{l} = \{x \in B \colon |f(x) - f(x_{n, n, 0)}| > l\}$ for $l \in \mathbb{N}$. In each of the two cases, $\{x_{n+1,1,0}\} \in B$ or not, by using (\ref{b2}) and (\ref{b3}), we get the following estimates: for $l=1, \dots, n-1$,
		\begin{equation}\label{b6}
		\frac{|E'_{2l}|}{|B|} \leq \frac{| T_{n-1} \cup \bigcup_{i=1}^{n-l} S_{n,i} |}{|T_{n}|}  \leq \frac{4}{(l+1)^{p_0}}, 
		\end{equation}
		for $l=n, \dots, n^2$,
		\begin{equation}\label{b7}
		\frac{|E'_{2l}|}{|B|} \leq \frac{| T_{n-1}|}{|T_{n}|}  \leq \frac{2}{n^{2p_0}},
		\end{equation}
		and, finally, for $l > n^2$,
		\begin{equation}\label{b8}
		|E'_{2l}| = 0.
		\end{equation}
		Moreover, recall the well known fact that for any $a \in \mathbb{C}$ we have
		\begin{equation}\label{b9}
		\sum_{x \in B} |f(x) - f_B|^p \leq 2^p \sum_{x \in B} |f(x) - a|^p.
		\end{equation}
		Therefore, by using (\ref{b6}), (\ref{b7}), (\ref{b8}) and (\ref{b9}) we obtain
		\begin{align}\label{b10}
		\begin{split}
		\frac{1}{|B|} \sum_{x \in B} |f(x) - f_B|^p & \leq \frac{2^p}{|B|} \sum_{x \in B} |f(x) - f(x_{n,n,1})|^p \\ &
		= \frac{2^p}{|B|} \int_0^\infty p \, \lambda^{p-1} |\{x \in B \colon |f(x) - f(x_{n,n,1})| > \lambda \}| \, d \lambda \\ &
		\leq \frac{p \, 2^{p+1} }{|B|} \sum_{l=0}^\infty (2l+2)^{p-1} |E'_{2l}|
		= p \, 4^p \sum_{l=0}^\infty \frac{(l+1)^{p-1} \cdot |E'_{2l}|}{|B|} \\ &
		\leq p \, 4^p \Big( 1 + \sum_{l=1}^{n-1} \frac{4 \cdot (l+1)^{p-1}}{(l+1)^{p_0}} + n^2 \cdot \frac{2(2n^2)^{p-1}}{n^{2p_0}} \Big) \\& 
		\leq  p \, 4^p \big( 1 + 4 \sum_{l=1}^\infty l^{p-p_0-1} + 2^p \big).
		\end{split}
		\end{align}
		Combining (\ref{b5}) and (\ref{b10}) shows that  
		\begin{displaymath}
		\sup_{B \subset X} \frac{1}{|B|} \sum_{x \in B} |f(x) - f_B|^p \leq C_p^p,
		\end{displaymath}
		independently of $f$, $\|f\|_\ast = 1$, and $B$.
		
		Now we prove that there exists $g \in BMO(\mathbb{X})$ such that $g \notin BMO^{p_0}(\mathbb{X})$. We start with a simple remark. Namely, for $f$ such that $|f(x) - f(y)| \leq 2$ for any neighboring points $x$ and $y$ and $B$ of the form $T_{n}$ or $T_{n} \cup \{x_{n+1, 1, 0}\}$, $n \geq 2$, the average value of $f$ over $B$ does not differ too much from $f(x_{n, n, 0})$. More precisely, by using (\ref{b2}), (\ref{b3}) and the estimate $|B| \geq b_{n}/2$ we get
		\begin{align}\label{b11}
		\begin{split}
		|f_B - f(x_{n, n, 0})| & \leq 2 + \frac{2}{|B|} \sum_{l=1}^{\infty} |\{x \in B \colon |f(x) - f(x_{n, n, 0})| > 2l \}| \\ &
		\leq 2 + \frac{2}{|T_{n}|} \Big( \sum_{l=1}^{n-1} |T_{n-1} \cup \bigcup_{i=1}^{n-l} S_{n,i}| + (n-1)^2 |T_{n-1}| \Big) \\ &
		\leq 2 + 2\sum_{l=1}^{n-1} \frac{|\bigcup_{i=1}^{n-l} S_{n,i}|}{|T_{n}|} + 2n^2 \frac{|T_{n-1}|}{|T_{n}|} \\&
		\leq 6 + 2\sum_{l=1}^{n-1} \frac{4}{(n-l)^{p_0}} \leq 6 + 8 \sum_{l=1}^{\infty} l^{-p_0} \leq N,
		\end{split}
		\end{align}
		for some fixed integer $N = N(p_0)$. 
		Now, take $g$ defined by the formula
		\begin{displaymath}
		g(x_{n,i,j}) = i + \sum_{k=1}^{n-1} k , \qquad j=0, \dots, m_{n,i}, \ i=1, \dots, n, \ n \in \mathbb{N}.
		\end{displaymath}
		It is easy to check that $g \in BMO(\mathbb{X})$ since for each $B \subset X$ at least one of the estimates (\ref{b5}) and (\ref{b10}) holds with $p$ replaced by $1$. Indeed, to obtain these inequalities for $f$ earlier we only used the information that $|f(x) - f(y)| \leq 2$ for any neighboring points $x$ and $y$. Our function $g$ satisfies this condition as well. Also (\ref{b11}) remains true if we put $g$ in place of $f$. Now, let $n \geq 2$ and take $B = T_{n}$. Observe that
		\begin{equation}\label{b12}
		|g(x) - g_B| \geq n - i - N, \qquad x \in S_{n, i}, \ i = 1, \dots, n.
		\end{equation}
		Therefore, if $n \geq 4N$, then by using (\ref{b4}) and (\ref{b12}) we have
		\begin{align}\label{b13}
		\begin{split}
		\frac{1}{|B|} \sum_{x \in B} |g(x) - g_B|^{p_0} & \geq \frac{1}{|B|} \sum_{l=1}^{\infty} p_0 \, (l-1)^{p_0-1} \, |\{x \in B \colon |g(x) - g_B| > l\}| \\ & \geq \frac{1}{|T_{n}|} \sum_{l=2}^{n-N-1} p_0 \, (l-1)^{p_0-1} \, \big|\bigcup_{i=1}^{n-N-l} S_{n,i} \big| \\ 
		& \geq \frac{p_0}{4}  \sum_{l=2}^{n-N-1} (l-1)^{p_0-1} \Big( \frac{1}{(N+l-1)^{p_0}} - \frac{1}{(n+1)^{p_0}} \Big) \\ &
		\geq \frac{p_0}{8}  \sum_{l=2}^{\lfloor n/2+3/2-N \rfloor} \frac{(l-1)^{p_0-1}}{(N+l-1)^{p_0}} \\ &
		\geq \frac{p_0}{2^{p_0+3}}  \sum_{l=N+1}^{\lfloor n/2+3/2-N \rfloor} (l-1)^{-1},
		\end{split}
		\end{align}
		since $(N+l-1)^{-p_0} \geq 2 (n+1)^{-p_0}$ for $l \leq \lfloor n/2+3/2-N \rfloor$ and $N+l-1 \leq 2(l-1)$ for $l \geq N+1$. Letting $n \rightarrow \infty$ we conclude that $g \notin BMO^{p_0}(\mathbb{X})$. $\raggedright \hfill \qed$
	\end{proof1}

At the end of this section we will be interested in test spaces $\mathbb{X}$ for which $BMO^p(\mathbb{X})$ coincides with $BMO(\mathbb{X})$ if and only if $p \leq p_0$ where $p_0 \in [1, \infty)$ is fixed. We can easily get such spaces slightly modifying the previous construction of $M$. Namely, instead of using (\ref{b1}) and (\ref{b2}), we define $m_{n, i}$ for $n \geq 2$ by

\begin{equation}
m_{n, i} = \Big\lfloor \frac{1}{\log(n)+1} \ \Big( \frac{b_{n} }{(n - i + 1)^{p_0}} - \frac{b_{n} }{(n - i + 2)^{p_0}} \Big) \Big\rfloor, \qquad i=1, \dots, n, \tag{C1'}\label{b1'}
\end{equation}
where $b_{n}$ is an even integer so large that
\begin{equation} \tag{C2'} \label{b2'}
|T_{n-1}| \leq \min \Big( \Big\lfloor \frac{1}{\log(n)+1} \ \Big( \frac{b_{n} }{(n + 1)^{p_0}} - \frac{b_{n} }{(n + 2)^{p_0}} \Big) \Big\rfloor, \frac{b_{n}}{n^{2p_0}}\Big). 
\end{equation}
We present a sketch of the proof of Proposition 2.

%\begin{proposition}
%	Fix $p_0 \geq 1$ and let $\mathbb{X} = (X, \rho, | \, \cdot \, |)$ be the test space with $M$ defined by using (\ref{b1'}) and (\ref{b2'}). Then for each $1 < p \leq p_0$ there exists $C_p > 0$ such that $\|f\|_{\ast, p} \leq C_p \|f\|_\ast$ for every $f \in BMO(\mathbb{X})$. On the other hand, there exists $g \in BMO(\mathbb{X})$ such that $g \notin BMO^{p}(\mathbb{X})$ for all $p > p_0$. 
%\end{proposition}

\begin{proof2}
For a fixed $p_0 \geq 1$ we let $\mathbb{X} = (X, \rho, | \, \cdot \, |)$ be the test space with $M$ defined by using (\ref{b1'}) and (\ref{b2'}). 
We show that for each $1 < p \leq p_0$ there exists $C_p > 0$ such that $\|f\|_{\ast, p} \leq C_p \|f\|_\ast$ for every $f \in BMO(\mathbb{X})$. To obtain this it suffices to observe that
\begin{displaymath}
p \, 4^p \Big( 1 + \frac{1}{\log(n)+1} \sum_{l=1}^{n-1} \frac{4 \cdot (l+1)^{p-1}}{(l+1)^{p_0}} + n^2 \cdot \frac{2(2n^2)^{p-1}}{n^{2p_0}} \Big),
\end{displaymath}
is bounded uniformly in $n$ if $p \leq p_0$. This allows us to get a proper variant of the estimate $(\ref{b10})$ for that $p$. 

Now we prove that for $g \in BMO(\mathbb{X})$ defined exactly in the same way as in the proof of Proposition 1 we have $g \notin BMO^{p}(\mathbb{X})$ for all $p > p_0$. To see this note that if $p > p_0$, then the estimates analogous to (\ref{b11}) and (\ref{b13}) remain true. Namely, for $B = T_{n}$ one can get
\begin{displaymath}
	|g_B - g(x_{n, n, 0})| \leq N,
\end{displaymath}
where $N$ is an integer independent of $n$, and
\begin{displaymath}
\frac{1}{|B|} \sum_{x \in B} |g(x) - g_B|^{p}
\geq \frac{p \, (\log(n)+1)^{-1}}{2^{p_0+3}}   \sum_{l=N+1}^{\lfloor n/2+3/2-N \rfloor} (l-1)^{p-p_0-1}.
\end{displaymath}
It is now clear that for $p > p_0$ the quantity on the right hand side tends to $\infty$ with $n \rightarrow \infty$. 
\end{proof2}

\section{Some related constructions}

In the last section we consider several variants of the discussed construction process in order to obtain test spaces with another interesting properties. Our first goal is to show that if the entries of the matrix $M$ grow fast enough, then the John--Nirenberg inequality holds for functions $f \in BMO(\mathbb{X})$. This result may be a little surprising at first, since we know that the John--Nirenberg inequality holds for every doubling metric measure spaces. Keeping that in mind, one may suppose that $\mathbb{X}$ should have rather little chance of preserving this property if we force the terms $m_{n,i}$ to grow fast. However, observe that in Section 5 the ratios between the values $m_{n,1}, \dots, m_{n,n}$ played a crucial role in estimating the mean oscillation of the studied functions and the obtained estimates were stronger for the smaller values of $m_{n,i} / m_{n,n}$, $i = 1, \dots, n-1$.

To formulate the next proposition in a more readable way it is convenient to identify the matrix $M$ with the sequence $M' = (m'_1, m'_2, \dots)$ formed by writing the entries of $M$ row by row, that is $M' = (m_{1,1}, m_{2,1}, m_{2,2}, m_{3,1}, \dots)$. In what follows, for simplicity, we use $M$ based on the geometric sequence $\{ 2^{k-1}\}_{k=1}^\infty$. Nevertheless, it will be clear that the presented proof also works for any lacunary sequence $\{m'_k\}_{k=1}^\infty$, that is a sequence satisfying $m'_{k+1}/m'_k \geq c$, $k \in \mathbb{N}$, for some fixed constant $c > 1$. 

\begin{proposition}
	Let $\mathbb{X} = (X, \rho, | \, \cdot \, |)$ be the test space with $M$ identified with the geometric sequence $\{ 2^{k-1}\}_{k=1}^\infty$. Then for the space $BMO(\mathbb{X})$ the John--Nirenberg inequality 
\begin{equation}\label{b14}
\frac{|\{x \in B \colon |f(x) - f_B| > \lambda \}|}{|B|} \leq c_1 \exp(-c_2 \lambda / \|f\|_\ast),
\end{equation}
holds with constants $c_1, c_2 >0$ independent of $f \in BMO(\mathbb{X})$, $B \subset X$ and $\lambda > 0$.
\end{proposition}

\begin{proof}
Let $f \in BMO(\mathbb{X})$ be such that $\|f\|_\ast = 1$. First, observe that the main difficulty in proving (\ref{b14}) is related to the situation in which $B$ as a set coincides with $T_{n}$ or $T_{n} \cup \{x_{n+1, 1, 0}\}$ for some $n \geq 2$. Indeed, for any other ball $B'$ we have $\max \{ |f(x) - f(y)| \colon x, y \in B'\} \leq 4$ and hence (\ref{b14}) with $B'$ in place of $B$ holds for any $\lambda > 0$ if we choose $c_1$ and $c_2$ such that $c_1 \exp(-4c_2) \geq 1$. Therefore, fix $n \geq 2$ and consider $B$ of the aforementioned form. Note that $2^{k} \leq |B| \leq 2^{k+1}$ where $k = \frac{n(n+1)}{2}$. Once again we will take advantage of the useful property that $|f(x) - f(y)| \leq 2$ for neighboring points $x$ and $y$. Proceeding just like we did before to get (\ref{b11}) we can estimate the value $|f_B - f(x_{n,n,0})|$ by some even integer $N$ which is independent of $f$, $n$ and the choice of $B$. Then for any integer $l \geq N$ we have
\begin{align*}
|\{x \in B \colon |f(x) - f_B| > 2 l \}| & \leq |\{x \in B \colon |f(x) - f(x_{n,n,0})| > 2(l - N/2) \}| \\ 
& \leq 2^{k-l+N/2+1} \\
&\leq 2^{N/2+1} 2^{-l} |B|,
\end{align*}
and now it is routine to choose $c_1$ and $c_2$ (independent of significant parameters) such that (\ref{b14}) holds for all $\lambda > 0$ and $B \subset X$ of an arbitrary form.
\end{proof}

For the presentation of the remaining two results we return to the matrix description of the space $\mathbb{X}$. We construct $M$ in a similar way as it was done earlier by using (\ref{b1}) and (\ref{b2}), but this time we choose the parameter $p_0$ separately in each step of induction. Namely, let $P = (p_2, p_3, \dots)$ be a sequence of numbers strictly bigger than $1$. We define $m_{n, i}$ for $n \geq 2$ by

\begin{equation} \tag{C1*} \label{b1*}
m_{n, i} = \Big\lfloor \frac{b_{n} }{(n - i + 1)^{p_{n}}} - \frac{b_{n} }{(n - i + 2)^{p_{n}}} \Big\rfloor, \qquad i=1, \dots, n, 
\end{equation}
where $b_{n}$ is an even integer so large that
\begin{equation} \tag{C2*} \label{b2*}
|T_{n-1}| \leq \min \Big( \Big\lfloor \frac{b_{n} }{(n + 1)^{p_{n}}} - \frac{b_{n} }{(n + 2)^{p_{n}}} \Big\rfloor, \frac{b_{n}}{n^{2p_{n}}}, \frac{b_{n}}{n^{n}}\Big). 
\end{equation}

Our next purpose will be to show that by a suitable choice of $P$ it is possible to obtain a space $\mathbb{X}$ for which the associated spaces $BMO^p(\mathbb{X})$ are all different. Although this result is not very revealing in view of Theorem 1, its advantage lies in the fact that the proof presented below, contrary to the proof of Theorem 1, is constructive. Namely, for each $1 \leq p_1 < p_2 < \infty$ we construct $f \in BMO^{p_1}(\mathbb{X}) \setminus BMO^{p_2}(\mathbb{X})$. In the following proposition we take $P$ formed by writing the elements of some countable dense subset of $(1, \infty)$ in an arbitrary order. We can use the set $\mathbb{Q} \cap (1, \infty)$, for example.

\begin{proposition}
	Let $P$ be the sequence defined as above and let $\mathbb{X} = (X, \rho, | \, \cdot \, |)$ be the test space with $M$ defined by using (\ref{b1*}) and (\ref{b2*}). Then for each $1 \leq p < p' < \infty$ there exists $g \in BMO^p(\mathbb{X})$ such that $g \notin BMO^{p'}(\mathbb{X})$.
\end{proposition}

\begin{proof}
	Fix $1 \leq p < p' < \infty$ and let $A = A(p,p') = [\frac{p+p'}{2}, p']$. We take $g$ defined by the formula
\begin{equation*}
g(x_{n,i,j}) = i \cdot \chi_A(p_n) + \sum_{k=1}^{n-1} k \cdot \chi_A(p_k), \qquad j=0, \dots, m_{n,i}, \ i=1, \dots, n, \ n \in \mathbb{N}.
\end{equation*}	
Note that $g$ is similar to the analogous function considered in the proof of Proposition 1, but this time it grows only in those $S_n$ for which the corresponding values $p_n$ belong to $A$. It is a standard procedure to show that $g \in BMO^p(\mathbb{X}) \setminus BMO^{p'}(\mathbb{X})$ and most of the work consists of proving the appropriate variants of the estimates (\ref{b10}), (\ref{b11}) and (\ref{b13}).
\end{proof}

We conclude our studies with an example of a test space $\mathbb{X}$ for which the associated spaces $BMO^p(\mathbb{X})$ coincide for the full range of the parameter $p$, but the John--Nirenberg inequality does not hold. Namely, we will prove the following.

\begin{proposition}
	There exists a (test) space $\mathbb{X}$ with the following properties:
	\begin{enumerate}[label=(\roman*)]
		\item for each $p > 1$ there exists $C_p > 0$ such that $\|f\|_{\ast, p} \leq C_p \|f\|_\ast$ for every $f \in BMO(\mathbb{X})$, 
		\item there exists $g \in BMO(\mathbb{X})$ such that for each $l \in \mathbb{N}$ we can find $B_l \subset X$ and $\lambda_l > 0$ satisfying
		\begin{equation*}
		\frac{|\{x \in B_l \colon |g(x) - g_{B_l}| > \lambda_l \}|}{|B_l|} > l \exp(- \lambda_l / l). 
		\end{equation*}
	\end{enumerate}
\end{proposition}
\begin{proof}
The space will be built by using $M$ constructed with an aid of (\ref{b1*}) and (\ref{b2*}) for some suitable sequence $P$ of positive integers. The key idea is to choose $P$ such that $p_n$ tends to $\infty$ very slowly. 

First, notice that the sole assumption $p_n \rightarrow \infty$ implies $(i)$. Indeed, let $f$ be such that $\|f\|_\ast = 1$. Observe that for each $p > 1$ there exists $N_0 = N_0(p) \geq 2$ such that $p_{n} \geq p + 1$ for all $n \geq N_0$. Therefore, (\ref{b10}) holds with $p + 1$ instead of $p_0$ for every $B$ of the form $T_{n}$ or $T_{n} \cup \{x_{n+1, 1, 0}\}$, $n \geq N_0$. Since for any other choices of $B$ there exists $K = K(p)$ independent of that $B$ (and $f$, of course) such that $\max\{|f(x) - f(y)| \colon x, y \in B\} \leq K$ we see that $(i)$ holds. 

It remains to show that with additional assumptions imposed on $P$ also $(ii)$ holds true. To be more specific slow growth of $p_n$ will suffice. Let $p_2 = 2$ and assume for convenience that $P$ is nondecreasing. We claim that there exists $N \in \mathbb{N}$ such that for any $f$, $\|f\|_\ast=1$, it holds $|f_B - f(x_{n,n,0})| \leq N$ for $B = T_{n}$, $n \geq 2$. Indeed, it suffices to see that now the estimate (\ref{b11}) with $p_0$ replaced by $2$ holds. We are ready to define $P$ inductively. Suppose that $p_n = l$ for some $n \geq 2$. We define $p_{n+1}$ by the formula
\begin{equation}\label{b15}
p_{n+1} = \left\{ \begin{array}{rl}
l & \textrm{if }  \frac{1}{4} \big( n^{-l} - (n+1)^{-l}\big) \leq l \, \exp(-(n-N-1)/l), \\

l+1 & \textrm{otherwise. }  \end{array} \right.
\end{equation}
Clearly, $p_n$ is nondecreasing and $p_n \rightarrow \infty$. 

Finally, take $g$ defined exactly in the same way as in the proof of Proposition 1. Of course, $g \in BMO(\mathbb{X})$. Fix $l \in \mathbb{N}$ such that $l \geq 2$ and let $n = n(l) = \max\{k \colon p_k = l\}$. Then by using (\ref{b4}) and (\ref{b15})

\begin{align*}
\frac{|\{ x \in T_{n} \colon |g(x)-g_{T_{n}}| \geq n-N-1 \}|}{|T_{n}|} & \geq \frac{|\{ x \in T_{n} \colon |g(x)-g(x_{n, n, 0})| \geq n-1 \}|}{|T_{n}|} \\
& \geq \frac{|S_{n,1}|}{|T_{n}|} \geq \frac{1}{4} \big( n^{-l} - (n+1)^{-l}\big) \\ & \geq l \, \exp(-(n-N-1)/l),
\end{align*}
and therefore we obtain that $(ii)$ holds for $B_k = T_{n}$ and $\lambda_l = n - N - 1$.
\end{proof}

\end{document}